\documentclass[a4paper,12pt]{amsart}
\usepackage{amsmath}
\usepackage{amsmath}
\usepackage{amssymb}
\usepackage{amscd}
\usepackage{mathrsfs}
\usepackage[all]{xy}
\usepackage[dvipdfm]{graphicx}
\def\mathcal{\mathscr}
\everymath{\displaystyle}
\newtheorem{thm}{Theorem}[section]
\newtheorem{lem}[thm]{Lemma}
\newtheorem{cor}[thm]{Corollary}
\newtheorem{prop}[thm]{Proposition}
\newtheorem{conj}[thm]{Conjecture}
\theoremstyle{definition}

\newtheorem{rem}[thm]{Remark}

\newtheorem{defn}[thm]{Definition}

\newcommand{\mca}[1]{{\mathcal{#1}}}

\def\Q{{\mathbb Q}}

\def\Z{{\mathbb Z}}
\def\R{{\mathbb R}}
\def\C{{\mathbb C}}

\def\CH{\mathrm{CH}\,}
\def\crit{\mathrm{Crit}\,}

\def\del{\partial}

\def\ep{\varepsilon}

\def\gcd{\mathrm{gcd}}
\def\GH{\mathrm{GH}}

\def\nice{\text{\rm nice}}

\def\pt{\text{\rm pt}}
\def\plim{\mathop{\lim_{\longleftarrow}}}

\def\SH{\text{\rm SH}\,} 
\def\sing{\text{\rm sing}\,} 
 
\def\spec{\text{\rm Spec}\,}

\begin{document}
\pagestyle{plain}
\thispagestyle{plain}

\title[A conjectural chain model for positive $S^1$-equivariant symplectic homology of star-shaped toric domains in $\mathbb{C}^2$]
{A conjectural chain model for positive $S^1$-equivariant symplectic homology of star-shaped toric domains in $\mathbb{C}^2$}

\author[Kei Irie]{Kei Irie}
\address{Research Institute for Mathematical Sciences, Kyoto University, Kyoto 606-8502, JAPAN} 
\email{iriek@kurims.kyoto-u.ac.jp} 
%\date{\today}

\begin{abstract}
For any star-shaped toric domain in $\mathbb{C}^2$, 
we define a filtered chain complex which conjecturally 
computes positive $S^1$-equivariant symplectic homology of the domain. 
Assuming this conjecture, we show that the limit $\lim_{k \to \infty} c^{\mathrm{GH}}_k(X)/k$
exists for any star-shaped toric domain $X \subset \mathbb{C}^2$, 
where $c^{\mathrm{GH}}_k$ denotes the $k$-th Gutt-Hutchings capacity.
\end{abstract} 

\maketitle

\section{Introduction} 

Let $n$ be a positive integer, 
and consider $\C^n$ with a symplectic form $\sum_{j=1}^n dx_j dy_j$. 
A star-shaped domain in $\C^n$ is a compact subset $X \subset \C^n$ with a $C^\infty$-boundary such that 
$(0,\ldots,0)$ is in the interior of $X$, 
and for any $z \in \C^n \setminus \{0\}$ 
the half line $\{ tz \mid t \in \R_{\ge 0}\}$ intersects $\del X$ transversally at a unique point. 

For any such $X$ and $-\infty < a < b \le \infty$, one can define a vector space $\SH^{S^1,[a,b)}_*(X)$ called $S^1$-equivariant symplectic homology. 
It is well-known that $\SH^{S^1, [\delta,\infty)}_*(X) \cong H^{S^1}_{*-(n+1)}(\pt)$ when $\delta>0$ is sufficiently close to $0$. 
On the other hand, this family of vector spaces (with maps between them) has rich quantitative information of $X$. 
In particular, one can define the Gutt-Hutchings capacities $(c^{\GH}_k)_{k \ge 1}$ 
from ``positive part'' of $S^1$-equivariant symplectic homology. 
The Gutt-Hutchings capacities were defined in \cite{Gutt_Hutchings} for Liouville domains. 
It is conjectured (\cite{Gutt_Hutchings} Conjecture 1.9) that 
the Gutt-Hutchings capacities coincide with 
the $S^1$-equivariant Ekeland-Hofer capacities \cite{Ekeland_Hofer} 
for compact star-shaped domains in $\C^n$. 

A star-shaped domain $X \subset \C^n$ is called a (star-shaped) toric domain if $X$ is invariant by the standard $T^n$-aciton on $\C^n$. 
When $X$ is a so called ``convex'' or ``concave'' toric domain, 
Gutt-Hutchings \cite{Gutt_Hutchings} proved explicit formulas to compute capacities $c^\GH_k(X)$ for all $k \ge 1$. 
One remarkable consequence of the formulas is that $\lim_{k \to \infty} \frac{c^\GH_k(X)}{k}$ exists 
if $X$ is a convex or concave toric domain. 
Actually, this existence of the limit holds under a much weaker assumption; see Remark 1.22 of \cite{Gutt_Hutchings}. 
The proof of the formulas in \cite{Gutt_Hutchings} 
is ``elementary'' in the sense that the proof uses only basic properties of the capacities, which are combined in a very clever way. 
On the other hand, it is not clear how to generalize the formulas for toric domains which are neither convex nor concave. 
Even for convex or concave domains, it is not clear
how to obtain information beyond the capacities, such as barcodes associated to persistent modules defined from $S^1$-equivariant symplectic homology. 

The aim of this note is to define a filtered chain complex for any 
star-shaped toric domain $X \subset \C^2$, 
which conjecturally computes $\SH^{S^1, [a,b)}_*(X)$ for any $0<a<b\le \infty$. 
Assuming this conjecture, we show that $\lim_{k \to \infty} \frac{c^\GH_k(X)}{k}$ exists for any 
star-shaped toric domain $X \subset \C^2$. 

Let us describe the plan of this paper. 
In Section 2, we define an $\R$-filtered chain complex $C^\Omega_*$ for any $\Omega \in \mca{S}^2$ 
(see Definition \ref{defn_starshaped} below). 
For any $\Omega \in \mca{S}^2$, we define a star-shaped domain $X_\Omega \subset \C^2$, 
formulate a conjecture that $C^\Omega_*$ computes $\SH^{S^1, [a,b)}_*(X_\Omega)$ for any $0 < a< b \le \infty$, 
and support this conjecture by some computations. 
In Section 3, we define a sequence of capacities $(c_k(\Omega))_{k \ge 1}$ for any $\Omega \in \mca{S}^2$. 
Assuming the above conjecture, one has $c_k(\Omega) = c^\GH_k(X_\Omega)$ for any $k \ge 1$ and $\Omega \in \mca{S}^2$. 
We compute the capacities $c_k(\Omega)$ when $\Omega$ is concave or (weakly) convex, 
and check that the results are consistent with the formulas in \cite{Gutt_Hutchings}. 
Moreover, we show that $\lim_{k \to \infty} \frac{c_k(\Omega)}{k}$ exists for any $\Omega \in \mca{S}^2$. 

{\bf Acknowledgement.} 
The author appreciates Jean Gutt and Michael Hutchings for very helpful comments on an earlier version of this paper. 
The author is supporeted by JSPS KAKENHI Grant Number 18K13407 and 19H00636. 

{\bf Conventions.} 
Throughout this paper we consider vector spaces over $\Q$ unless otherwise specified. 
An $\R$-filtration on a vector space $V$ is a family of subspaces $(V^a)_{a \in \R}$ such that $a \le b \implies V^a \subset V^b$. 
We set $V^\infty:=V$ and $V^{-\infty}:=0$. 
For any $a<b$, we denote $V^{[a,b)}:= V^b/V^a$. 

\section{A chain model}

\subsection{Definition of a chain model} 

Let us start with the following definition. 

\begin{defn}\label{defn_starshaped} 
For any $n \in \Z_{\ge 1}$, let
$\Sigma^n:= \{ v \in ( \R_{\ge 0})^n \mid |v|=1 \}$.
Let $\mca{S}^n$ denote the set consisting of 
$\Omega \subset (\R_{\ge 0})^n$ such that 
there exists $r_\Omega \in C^\infty(\Sigma^n, \R_{>0})$ satisfying 
$\Omega = \{ tz \mid 0 \le t \le r_\Omega(z), \, z \in \Sigma^n \}$.
For any $\Omega \in \mca{S}^n$, let $U_\Omega:= (\R_{>0})^n \setminus \Omega$, 
and let $\bar{U}_\Omega \subset (\R_{\ge 0})^n$ denote the closure of $U_\Omega$
in $(\R_{\ge 0})^n$. 
\end{defn} 

In this paper we mostly consider the case $n=2$. 
For any $\Omega \in \mca{S}^2$, we define a $\Z$-graded $\Q$-vector space $C^\Omega_*$ by 
\[ 
C^\Omega_*:= \bigoplus_{(m_1, m_2) \in \Z^2 \setminus (\Z_{\le 0})^2}   C^{\sing}_{*+1-2(m_1+m_2)}(U_\Omega) \otimes  H_*(S^1),
\] 
where $C^{\sing}_*$ denotes the singular chain complex and $S^1:=\R/\Z$. 

Let us define a boundary operator on $C^\Omega_*$. 
Let $e_0:= [\pt] \in H_0(S^1)$ and $e_1:=[\sigma] \in H_1(S^1)$, 
where $\sigma:[0,1] \to S^1; \, t \mapsto [t]$. 
For any homogeneous element $x \in C^\Omega_*$, 
let us set 
\[ 
x= \sum_{(m_1, m_2, i) \in (\Z^2 \setminus (\Z_{\le 0})^2) \times \{0, 1\}} x_{m_1, m_2, i} \otimes e_i, 
\] 
and define $\del x$ by 
\begin{align*} 
(\del x )_{m_1, m_2, 0 }&:=  \del^\sing x _{m_1, m_2, 0},  \\ 
(\del x)_{m_1, m_2, 1}&:= \del^\sing x_{m_1, m_2, 1} + 
(-1)^{|x|}  ( m_2 \cdot  x_{m_1+1, m_2, 0} -  m_1 \cdot x_{m_1, m_2+1, 0}),
\end{align*} 
where $\del^\sing$ denotes the boundary operator of the singular chain complex. 
One can check $\del^2=0$ by direct computations.

Let us define an $\R$-filtration on $C^\Omega_*$. 
For any $(x_1, x_2) \in \R^2$, define 
$A_{x_1, x_2}: \R^2 \to \R$ by 
\[ 
A_{x_1,x_2}(y_1,y_2):= x_1 y_1 + x_2 y_2. 
\] 
For any $a \in \R$ and $(m_1,m_2) \in \Z^2$, let 
\[ 
U_\Omega(a: m_1,m_2):= \{(x_1,x_2) \in U_\Omega \mid A_{m_1,m_2}(x_1,x_2) < a\}. 
\] 
and 
\[ 
C^{\Omega,a}_*:= \bigoplus_{(m_1, m_2) \in \Z^2 \setminus (\Z_{\le 0})^2}  
C^{\sing}_{*+1-2(m_1+m_2)}(U_\Omega(a:m_1,m_2)) \otimes  H_*(S^1).
\] 
Then $(C^{\Omega, a}_*)_{a \in \R}$ is an $\R$-filtration on $C^\Omega_*$. 
For any $a<b$, we denote $C^{\Omega, [a,b)}_*:= C^{\Omega, b}_*/C^{\Omega, a}_*$. 
For any $(m_1, m_2) \in \Z^2 \setminus (\Z_{\le 0})^2$, there holds
\begin{equation}\label{U_Omega_length} 
U_\Omega(a:m_1+1, m_2) , U_\Omega(a: m_1, m_2+1) \subset U_\Omega(a: m_1, m_2). 
\end{equation}
Thus $\del(C^{\Omega, a}_*) \subset C^{\Omega, a}_{*-1}$. 

If $\Omega_1, \Omega_2 \in \mca{S}^2$ satisfy $\Omega_1 \subset \Omega_2$, 
then $U_{\Omega_2} \subset U_{\Omega_1}$. 
Then we obtain a natural chain map 
$C^{\Omega_2}_* \to C^{\Omega_1}_*$
which preserves the $\R$-filtrations. 

\begin{lem}\label{lem_F_m} 
Let $-\infty < a<b \le \infty$. 
For any $m \in \Z$, let 
\[ 
F_mC^{\Omega, [a,b)}_*: = \bigoplus_{\substack{(m_1, m_2) \in \Z^2 \setminus (\Z_{\le 0})^2 \\ m_1+m_2 \le m}}
C^{\sing}_{*+1-2(m_1+m_2)}(U_\Omega(b:m_1,m_2), U_\Omega(a:m_1,m_2)) \otimes  H_*(S^1).  
\] 
Then $(F_mC^{\Omega, [a,b)}_*)_{m \in \Z}$ is a filtration on $C^{\Omega, [a,b)}_*$. 
Let $(E^r, \del_{E^r})_{r \ge 1}$ be the spectral sequence associated to this filtration. 
Then the following holds. 
\begin{itemize} 
\item[(i):] There exists an isomorphism 
\[ 
E^1_{p,q}
\cong \bigoplus_{\substack{m_1+m_2=p \\ i+j=q-p+1}}
H_i(U_\Omega(b:m_1,m_2), U_\Omega(a:m_1,m_2)) \otimes  H_j(S^1) 
\] 
such that 
$\del_{E^1}: E^1_{p,q} \to E^1_{p-1,q}$ is given by 
\begin{align*} 
(\del_{E^1} x )_{m_1, m_2, 0 }&=0, \\ 
(\del_{E^1} x)_{m_1, m_2, 1}&=(-1)^{|x|}  ( m_2 \cdot  x_{m_1+1, m_2, 0} -  m_1 \cdot x_{m_1, m_2+1, 0}). 
\end{align*} 
\item[(ii):] 
$\del_{E^r}=0$ if $r \ge 2$. 
Moreover $H_*(C^{\Omega,[a,b)}) \cong \bigoplus_{p+q=*} E^\infty_{p,q}$. 
\end{itemize} 
\end{lem}
\begin{proof} 
(i) is straightforward. 
To see (ii), for each $l \in \Z$ let 
\begin{align*} 
C^l_*:=  & \bigoplus_{m_1+m_2=l} C_{*+1-2l}^{\sing}(U_\Omega(b:m_1,m_2), U_\Omega(a:m_1,m_2)) \otimes \R e_1  \\ 
\oplus &\bigoplus_{m_1+m_2=l+1} C_{*-1-2l}^{\sing}(U_\Omega(b:m_1,m_2), U_\Omega(a:m_1,m_2)) \otimes \R e_0. 
\end{align*} 
Then $C^l_*$ is a subcomplex of $C_*:= C^{\Omega,[a,b)}_*$, and 
there holds $C_* = \bigoplus_{l \in \Z} C^l_*$, 
in particular $H_*(C) \cong \bigoplus_{l \in \Z} H_*(C^l)$. 
Each $C^l$ is equipped with a filtration 
$F^mC^l:= F^m C \cap C^l\, (m \in \Z)$. 
Let $(E^r(C^l))_{r \ge 1}$ be the spectral sequence associated to this filtration. 
Since $F^m C \cong \bigoplus_{l \in \Z} F^m C^l$ for each $m$, 
there holds $E^r_{p,q}(C) \cong \bigoplus_{l \in \Z} E^r_{p,q}(C^l)$
for any $r \ge 1$. 
Thus it is sufficient to show that 
$\del_{E^r(C^l)}=0 \,(r \ge 2)$ and 
$H_*(C^l) \cong \bigoplus_{p+q=*} E^\infty_{p,q}(C^l)$ 
for each $l \in \Z$. 
This follows from $F^{l-1}C^l=0$ and $F^{l+1}C^l =C^l$. 
\end{proof} 

\begin{rem}\label{rem_E_1} 
Suppose that $H_i(U_\Omega(b:m_1,m_2), U_\Omega(a:m_1,m_2)) \ne 0$ only if $i=0$. 
Then $E^1_{p,q} \ne 0$ only if $q=p$ or $q=p-1$. Moreover, for any $j \in \Z$ 
\[ 
E^1_{p, p-1+j} \cong \bigoplus_{m_1+m_2=p} H_0(U_\Omega(b:m_1,m_2), U_\Omega(a:m_1,m_2))  \otimes H_j(S^1). 
\] 
\end{rem}

\subsection{Conjectural relation to $S^1$-equivariant symplectic homology} 

Let $n$ be a positive integer, and let 
$\lambda_0:= \frac{1}{2} \sum_{j=1}^n (x_j dy_j - y_j dx_j) \in \Omega^1(\C^n)$. 
For any star-shaped domain $X \subset \C^n$, 
$(X, \lambda_0)$ is a Liouville domain. 

For any $-\infty < a < b \le \infty$, 
one can define a $\Z$-graded vector space 
$\SH^{S^1, [a,b)}_*(X, \lambda_0)$, 
which we abbreviate by $\SH^{S^1,[a,b)}_*(X)$, 
called $S^1$-equivariant symplectic homology. 
The family of vector spaces $(\SH^{S^1, [a,b)}_*(X))_{a,b,X}$ is equipped with 
the maps (transfer morphisms) 
\[ 
\SH^{S^1, [a,b)}_*(X) \to \SH^{S^1, [a', b')}_*(X') 
\] 
for any $(a,b,X)$ and $(a', b', X')$ such that 
$a \le a'$, $b \le b'$ and $X' \subset X$.

\begin{rem} 
$S^1$-equivariant symplectic homology was defined by Viterbo \cite{Viterbo_GAFA}. 
Bourgeois-Oancea \cite{Bourgeois_Oancea} gave alternative definitions via family Floer homology 
following Seidel \cite{Seidel_biased}. 
Gutt-Hutchings \cite{Gutt_Hutchings} uses a family Floer homology definition, following the treatment in Gutt \cite{Gutt}. 
\end{rem}

For any $\Omega \in \mca{S}^2$, 
\[ 
X_\Omega:= \{ (z_1, z_2) \in \C^2 \mid (\pi|z_1|^2, \pi |z_2|^2) \in \Omega \} 
\] 
is a star-shaped domain in $\C^2$. 
Now we can state the following conjecture. 

\begin{conj}\label{conj_main} 
For any $\Omega \in \mca{S}^2$ and $0<a<b  \le \infty$, 
one can define an isomorphism of $\Z$-graded vector spaces 
\[ 
F^{a,b}_\Omega: H_*(C^{\Omega, [a,b)}) \cong \SH^{S^1, [a,b)}_*(X_\Omega) 
\] 
so that the diagram 
\[ 
\xymatrix{ 
H_*(C^{\Omega, [a,b)})\ar[d]_-{F^{a,b}_\Omega}^-{\cong} \ar[r]& H_*(C^{\Omega',[a',b')})  \ar[d]^-{F^{a',b'}_{\Omega'}}_-{\cong} \\
\SH^{S^1,[a,b)}_*(X_\Omega) \ar[r] & \SH^{S^1, [a', b')}_*(X_{\Omega'}) 
}
\] 
commutes for any $a, b, \Omega$ and $a', b', \Omega'$ such that 
$a \le a'$, $b \le b'$ and $\Omega' \subset \Omega$. 
\end{conj} 

Let us briefly explain an idea to obtain the conjectural isomorphism $F^{a, b}_\Omega$. 
Given $0<a<b<\infty$ such that $a,b \not\in \spec(\Omega)$ and a positive integer $N$, 
take an autonomous Hamiltonian $H$ on $\C^2$ so that 
$\SH^{S^1, [a,b)}_{\le N}(X_\Omega) \cong \mathrm{HF}^{S^1, [a,b)}_{\le N}(H)$
and the following property holds: 
every $1$-periodic orbit $\gamma$ of $X_H$ with 
$\mathrm{ind}_{\mathrm{CZ}}(\gamma) \le N$
and 
$\mca{A}_H(\gamma)>0$ is contained in $(\C \setminus \{0\})^2$. 
Here $\mathrm{ind}_{\mathrm{CZ}}$ denotes the Conley-Zehnder index and 
$\mca{A}_H$ denotes the Hamiltonian action functional. 

To compute $\mathrm{HF}^{S^1, [a,b)}_{\le N}(H)$ we take $\ep>0$ 
and consider an almost complex structure $J^\ep$ on $\C^2$ which satisfies  
\[ 
J^\ep(\del_{\theta_i}) = - \ep r_i \del_{r_i}  \qquad (z_i = e^{r_i + \sqrt{-1} \theta_i}, \, i=1, 2) 
\] 
on the complement of a neighborhood of $\C \times \{0\}  \cup \{0\} \times \C$. 
Conjecturally, 
when $\ep$ is sufficiently close to $0$, 
$\mathrm{HF}^{S^1, [a,b)}_{\le N}(H)$
can be computed by counting certain Morse trajectories on $(\R_{>0})^2$, 
and one obtains an isomorphism $\mathrm{HF}^{S^1, [a,b)}_{\le N}(H) \cong H_{\le N}(C^{\Omega, [a,b)})$ via finite-dimensional Morse theory; 
this gives $F^{a,b}_\Omega$ up to degree $N$.

\subsection{Computations of relative homologies} 

In this subsection we compute relative homologies $H_*(C^{\Omega, [a,b)})$ for some special cases, 
verifying that Conjecture \ref{conj_main} is consistent with known properties of $S^1$-equivariant symplectic homology.  

We start with some preparations on toric star-shaped domains in $\C^2$. 
For any $\Omega \in \mca{S}^2$, let 
us define $\rho_\Omega \in C^\infty([0,\pi/2], \R_{>0})$ 
by $\rho_\Omega(\theta):=r_\Omega(\cos \theta, \sin \theta)$. 
In other words, 
\begin{equation}\label{eqn_Omega} 
\Omega = \{ (r \cos \theta, r \sin \theta) \mid 0 \le \theta \le \pi/2, \, 0 \le r \le \rho_\Omega(\theta) \}. 
\end{equation}
Let us define 
\begin{align*} 
\del \Omega&:= \{ (\rho_\Omega(\theta) \cos \theta, \rho_\Omega(\theta) \sin \theta) \mid 0 \le \theta \le \pi/2 \}, \\
\del_+\Omega&:= \del\Omega \cup \{ (t,0) \mid t \ge \rho_\Omega(0)\} \cup \{ (0,t) \mid t \ge \rho_\Omega(\pi/2)\}. 
\end{align*} 
For any $c \in \R$ and $(m_1,m_2) \in \Z^2 \setminus (\Z_{\le 0})^2$, let 
\begin{align*} 
\bar{U}_\Omega(c:m_1,m_2)&:= \{(x_1, x_2) \in \bar{U}_\Omega \mid A_{m_1,m_2}(x_1,x_2)<c\}, \\
\del_+\Omega(c:m_1,m_2)&:= \{(x_1,x_2) \in \del_+ \Omega \mid A_{m_1,m_2}(x_1,x_2)<c\}. 
\end{align*} 

\begin{lem}\label{lem_barU}
For any $\Omega \in \mca{S}^2$, $c \in \R_{>0}$ and $(m_1,m_2) \in \Z^2 \setminus (\Z_{\le 0})^2$, 
\[ 
H_*(\bar{U}_\Omega(c:m_1,m_2), U_\Omega(c:m_1,m_2)) = 
H_*(\bar{U}_\Omega(c:m_1,m_2), \del_+\Omega(c:m_1,m_2))=0. 
\] 
\end{lem} 
\begin{proof} 
Let us define 
\[ 
\mca{C}:= \{ \text{critical values of $A_{m_1,m_2}|_{\del \Omega}$ } \}  \cup \{ m_1 \rho(0), m_2 \rho(\pi/2) \}. 
\] 
Note that $\mca{C}$ is a null set. 
We may assume $c \not\in \mca{C}$, 
since the case $c \in \mca{C}$ follows from this case by taking limits. 
Then $\bar{U}_\Omega(c: m_1, m_2)$ is a manifold with corners, 
which implies that 
$H_*(\bar{U}_\Omega(c:m_1,m_2), U_\Omega(c:m_1,m_2))=0$. 
Next we show that 
$H_*(\bar{U}_\Omega(c:m_1,m_2), \del_+\Omega(c:m_1,m_2))=0$. 
Let 
\[ 
\del \Omega(c: m_1, m_2):= \{ (x_1, x_2) \in \del \Omega \mid A_{m_1, m_2}(x_1, x_2) < c\}. 
\] 
Then it is sufficient to show that the inclusion maps 
\[ 
i: \del \Omega(c:m_1,m_2) \to \del_+ \Omega (c:m_1, m_2), \qquad 
j: \del \Omega(c:m_1,m_2) \to \bar{U}_\Omega(c:m_1,m_2) 
\] 
are homotopy equivalent maps. 
For any $x \in \bar{U}_\Omega(c:m_1,m_2)$, 
let $\rho(x) \in \R_{>0}$ be the unique positive real number such that 
$\rho(x) x \in \del \Omega$. Then 
\[ 
r: \bar{U}_\Omega(c:m_1,m_2) \to \del \Omega(c:m_1,m_2); \, x \mapsto \rho(x) \cdot x
\] 
is a homotopy inverse of $j$, and $r|_{\del_+ \Omega(c:m_1,m_2)}$ is an inverse of $i$.
\end{proof} 

For any $\Omega \in \mca{S}^2$, let us define 
\[ 
P(\Omega):= \{ (\rho_\Omega(0), 0), (0, \rho_\Omega(\pi/2))\} \cup \bigcup_{(m_1, m_2) \in \Z^2 \setminus (\Z_{\le 0})^2}  \crit_+(A_{m_1,m_2}|_{\del \Omega}), 
\] 
where
\[ 
\crit_+(A_{m_1,m_2}):= \{ p \in \del  \Omega \mid d A_{m_1,m_2}|_{\del \Omega}(p)=0, \quad  A_{m_1,m_2}(p) > 0 \}. 
\] 
For any $p \in P(\Omega)$ and $m \in \Z_{>0}$, we define 
$A(p,m) \in \R_{>0}$ and $i(p,m) \in \Z$ as follows: 
\begin{itemize}
\item If $p=(\rho_\Omega(0), 0)$, 
\[ 
A(p,m):= m \cdot \rho_\Omega(0) , \quad
i(p,m):= 1 + 2(m+ [mt_1]), 
\] 
where $t_1 \in \R$ is defined so that
$T_p(\del\Omega)$ is generated by $(-t_1, 1)$. 
\item If $p= (0, \rho_\Omega(\pi/2))$,
\[  
A(p,m):= m \cdot \rho_\Omega(\pi/2), 
 \quad
i(p,m):= 1 + 2(m+ [mt_2]),
\] 
where $t_2 \in \R$ is defined so that 
$T_p(\del\Omega)$ is generated by $(1,-t_2)$. 
\item If $p \not\in \{ (\rho_\Omega(0), 0), (0, \rho_\Omega(\pi/2))\}$, 
there exists unique $(m_1, m_2) \in \Z^2 \setminus (\Z_{\le 0} )^2$ such that 
$p \in \crit_+(A_{m_1,m_2})$ and $m= \gcd(m_1, m_2)$. 
Let $\mu(p)$ denote the Morse index of $p$ as a critical point of $A_{m_1,m_2}|_{\del \Omega}$. 
Then 
\[ 
A(p,m):=A_{m_1,m_2}(p), \quad 
i(p,m):= 2(m_1+m_2) + \mu(p) - 1. 
\] 
\end{itemize} 
Let 
$\spec(\Omega):= \{ A(p,m) \mid  (p,m) \in P(\Omega) \times \Z_{>0} \} \subset \R_{>0}$. 
It is easy to see that $\spec(\Omega)$ is of measure zero and closed in $\R_{\ge 0}$, 
in particular $\min \spec(\Omega)$ exists and is positive. 

Let $\mca{S}^2_\nice$ denote the set 
consists of $\Omega \in \mca{S}^2$
satisfying the following conditions: 
\begin{itemize} 
\item For any $(m_1,m_2) \in \Z^2 \setminus (\Z_{\le 0})^2$, 
$(\rho_\Omega(0), 0)$ and $(0, \rho_\Omega(\pi/2))$ are not critical points of $A_{m_1,m_2}|_{\del \Omega}$.
Moreover $A_{m_1, m_2}|_{\del \Omega}$ is a Morse function, i.e. every critical point of $A_{m_1,m_2}|_{\del \Omega}$ is nondegenerate. 
\item If $(p,m), (p', m') \in P(\Omega) \times \Z_{>0}$ satisfy 
$A(p,m)=A(p',m')$, then $(p,m) = (p',m')$. 
\end{itemize} 
It is easy to see that $\mca{S}^2_\nice$ is residual in $\mca{S}^2$ with the $C^\infty$-topology; 
see \cite{Irie_equidistribution} Lemma 6.1. 

\begin{prop}\label{prop_01} 
For any $\Omega \in \mca{S}^2$ and $0<a<b \le \infty$ such that
$a,b \not\in \spec(\Omega)$, the following holds. 
\begin{itemize} 
\item[(i):] If $[a,b)  \cap \spec(\Omega) = \emptyset$, then $H_*(C^{\Omega,[a,b)})=0$. 
\item[(ii):] If $\Omega \in \mca{S}^2_\nice$ and 
$[a,b) \cap \spec(\Omega)$ consists of one element $A(p,m)$ with 
$p \in \{(\rho_\Omega(0), 0), (0, \rho_\Omega(\pi/2))\}$, then 
$H_*(C^{\Omega, [a,b)}) \cong H_{*-i(p,m)}(\pt)$. 
\item[(iii):] If $\Omega \in \mca{S}^2_\nice$ and 
$[a,b)  \cap \spec(\Omega)$ consists of one element $A(p,m)$ with
$p \not\in \{(\rho_\Omega(0), 0), (0, \rho_\Omega(\pi/2))\}$, then
$H_*(C^{\Omega, [a,b)})  \cong H_{*-i(p,m)}(S^1)$. 
\end{itemize} 
\end{prop} 
\begin{proof} 
Let us consider the filtration $(F_m)_m$ on $C^{\Omega, [a,b)}_*$ 
and the associated spectral sequence as in Lemma \ref{lem_F_m}. 

(i): 
By Lemma \ref{lem_barU} and the assumption, 
\[ 
H_*(U_\Omega(b:m_1,m_2), U_\Omega(a:m_1,m_2)) 
\cong 
H_*(\del_+\Omega(b:m_1, m_2),  \del_+\Omega(a:m_1,m_2)) 
=0 
\] 
for any 
$(m_1,m_2) \in \Z^2 \setminus (\Z_{\le 0})^2$, thus $E^1_{p,q}=0$ for any $(p,q) \in \Z^2$. 

(ii): Let us consider the case $p=(\rho_\Omega(0), 0)$. 
By Lemma \ref{lem_barU} and the assumption, 
\[ 
H_*(U_\Omega(b:m_1,m_2), U_\Omega(a:m_1,m_2)) 
\cong \begin{cases}  H_*(\pt) &(m_1=m, \, m_2 > mt_1),   \\ 0 &(\text{otherwise}). \end{cases}
\] 
Then, by Lemma \ref{lem_F_m} and Remark \ref{rem_E_1}, we obtain 
\[ 
E^2_{k,l} \cong \begin{cases} \Q &(k=m+[mt_1]+1, \, l=m+[mt_1] ) \\ 0&(\text{otherwise}) \end{cases}
\] 
and $\del_{E^2}=0$, which implies 
$H_*(C) \cong H_{*-i(p,m)}(\pt)$. 
The case $p=(0, \rho_\Omega(\pi/2))$ is similar and omitted. 

(iii): 
By Lemma \ref{lem_barU} and the assumption, 
\[ 
H_*(U_\Omega(b:m_1,m_2), U_\Omega(a:m_1,m_2)) 
\cong 
\begin{cases} H_{*-\mu(p)} (\pt) &(p \in \crit_+(A_{m_1,m_2}), \, m = \gcd(m_1, m_2)),  \\
0&( \text{otherwise}). 
\end{cases} 
\] 
By Lemma \ref{lem_F_m}, we obtain 
\[ 
E^1_{k,l} \cong \begin{cases} \Q &(k=m_1+m_2, \, l - k - \mu(p) \in \{-1, 0\}) \\ 0 &(\text{otherwise}) \end{cases}
\] 
and $\del_{E^1}=0$, which implies 
$H_*(C^{\Omega, [a,b)}) \cong H_{*-i(p,m)}(S^1)$. 
\end{proof} 

For any $0<a \le \infty$, let 
$H^{+,a}_*(\Omega):= \plim_{\delta \to 0+} H_*(C^{\Omega, [\delta,a)})$. 
By Proposition \ref{prop_01} (i), 
$H^{+,a}_*(\Omega) \to H_*(C^{\Omega,[\delta,a)})$ is an isomorphism
if $\delta \in (0, \min \spec(\Omega))$. 

\begin{prop}\label{prop_plus} 
\begin{itemize} 
\item[(i):] For any $\Omega \in \mca{S}^2$, there holds
$H^{+,\infty}_*(\Omega) \cong H^{S^1}_{*-3}(\pt)$. 
\item[(ii):] For any $\Omega, \Omega' \in \mca{S}^2$
such that $\Omega' \subset \Omega$, 
the natural map 
$H^{+,\infty}_*(\Omega) \to H^{+,\infty}_*(\Omega')$ is an isomorphism. 
\end{itemize} 
\end{prop} 
\begin{proof} 
(i): 
Take $\delta \in (0, \min \spec(\Omega))$. 
Then $H^{+,\infty}_*(\Omega) \cong H_*(C^{\Omega, [\delta,\infty)})$. 
By Lemma \ref{lem_barU}, 
it is easy to show that for any 
$(m_1,m_2) \in \Z^2 \setminus (\Z_{\le 0})^2$ 
\[ 
H_*(U_\Omega, U_\Omega(\delta: m_1, m_2))
\cong
\begin{cases} 
H_*(\pt) &( (m_1, m_2) \in (\Z_{>0})^2),  \\
0 &(\text{otherwise}). 
\end{cases} 
\]

Consider the filtration $(F_m)_{m \in \Z}$ on $C^{\Omega, [\delta, \infty)}_*$ as before. 
Then, by Lemma \ref{lem_F_m} and Remark \ref{rem_E_1}, 
$E^1_{p,q} \ne 0$ only if $q-p \in \{0,-1\}$, and there exists an isomorphism
\[ 
E^1_{p, p-1+j} \cong \bigoplus_{\substack{m_1+m_2=p \\ m_1,m_2>0}} H_0(\pt) \otimes H_j(S^1). 
\] 
For any $p \ge 2$ one has an exact sequence
\[ 
\xymatrix{ 
0 \ar[r] & \Q \ar[r]_-{\star}& E^1_{p,p-1} \ar[r]_-{\del_{E^1}} & E^1_{p-1, p-1} \ar[r]& 0, 
} 
\] 
where $\star$ maps $1$ to $\sum_{j=1}^{p-1} x_{j, p-j, 0} \otimes e_0$ 
such that $x_{j, p-j, 0} \ne 0$ for any $j$. 
Hence we obtain 
\[ 
E^2_{p,q} \cong \begin{cases} \Q &(p \ge 2, \, q = p-1), \\   0 &(\text{otherwise}). \end{cases}
\] 
This implies $H_*(C^{\Omega, [\delta, \infty)}) \cong H^{S^1}_{*-3}(\pt)$. 

(ii): The natural chain map 
$C^{\Omega, [\delta,\infty)}_* \to C^{\Omega', [\delta,\infty)}_*$ 
respects the filtrations $(F_m)_{m \in \Z}$ and gives isomorphisms on $E^1$-pages. 
\end{proof} 

The next corollary follows from the above proof of Proposition \ref{prop_plus} (i). 

\begin{cor}\label{cor_plus} 
For any $\delta \in (0, \min \spec(\Omega))$ and $k \ge 1$, 
any element of $H_{2k+1}(C^{\Omega, [\delta, \infty)}) \cong \Q$
is represented by 
$x= \sum_{(m_1, m_2, i) \in (\Z^2 \setminus (\Z_{\le 0})^2) \times \{0, 1\}} x_{m_1, m_2, i} \otimes e_i \in C^{\Omega, [\delta,\infty)}_{2k+1}$
such that: 
\begin{itemize} 
\item $x_{m_1,m_2,i}=0$ unless $m_1,m_2>0$, $m_1+m_2=k+1$ and $i=0$. 
\item $x_{j, k+1-j, 0} = a_j[p] \otimes e_0$ for any $1 \le j \le k$,
where $p \in U_\Omega$ and $(a_1, \ldots, a_k)$ satisfies 
$(k-j) \cdot a_{j+1} - j \cdot a_j=0$ for any $1 \le j \le k-1$. 
\end{itemize} 
\end{cor}  

\section{Capacities} 

\subsection{Definition and basic properties} 

For any $\Omega \in \mca{S}^2$, 
we define a sequence $(c_k(\Omega))_{k \ge 1}$ as follows. 
For any $a \in \R_{>0}$, let 
\[ 
(i^a_\Omega)_*: H^{+,a}_*(\Omega) \to H^{+,\infty}_*(\Omega)
\] 
be  the natural map. For any $k \ge 1$, let 
\[ 
c_k(\Omega):= \inf\{ a \mid  (i^a_\Omega)_{2k+1} \ne 0 \}. 
\] 

\begin{prop}\label{prop_property_c_k}
The following holds for any $\Omega \in \mca{S}^2$ and $k \ge 1$. 
\begin{itemize} 
\item[(i):] For any $\Omega' \in \mca{S}^2$ such that $\Omega' \subset \Omega$, 
there holds $c_k(\Omega') \le c_k(\Omega)$. 
\item[(ii):] For any $c \in \R_{>0}$, $c_k(c\Omega)=c \cdot c_k(\Omega)$. 
\item[(iii):] $c_k(\Omega) \in \spec(\Omega)$.
\item[(iv):] $c_k(\Omega) \le c_{k+1}(\Omega)$. 
\end{itemize} 
\end{prop} 
\begin{proof} 
(i): For any $0 < a \le \infty$, let $j^a: H^{+,a}_*(\Omega) \to H^{+,a}_*(\Omega')$ be the natural map. 
Then $c_k(\Omega') \le c_k(\Omega)$ since 
$i^a_{\Omega'} \circ j^a=j^\infty \circ i^a_\Omega$ 
and $j^\infty$ is an isomorphism by Proposition \ref{prop_plus} (ii). 

(ii) follows from the isomorphism $H^{+,a}_*(\Omega) \cong H^{+, ca}_*(c\Omega)$ defined for any $0<a \le \infty$ and $c>0$, 
which is defined by the scaling diffeomorphism $U_\Omega \to U_{c\Omega}; \, (x_1,x_2) \mapsto (cx_1, cx_2)$. 

(iii) follows from Proposition \ref{prop_01} (i) and $\spec(\Omega)$ is a closed set. 

(iv): Let us define a linear map $u:C^\Omega_* \to C^\Omega_{*-2}$ by 
\[ 
(ux)_{m_1,m_2, i}: = x_{m_1+1, m_2, i} + x_{m_1, m_2+1, i} \qquad(i=0, 1). 
\] 
By direct computations one can check that $u$ commutes with the boundary map on $C^\Omega$. 
$u$ respects the $\R$-filtration on $C^\Omega$ by (\ref{U_Omega_length}). 
Also Corollary \ref{cor_plus} implies that 
$H_*(u): H_{2k+3}(C^{\Omega, [\delta,\infty)}) \to H_{2k+1}(C^{\Omega, [\delta,\infty)})$ 
is an isomorphism for any $k \ge 1$, 
which implies that $c_k(\Omega) \le c_{k+1}(\Omega)$. 
\end{proof} 

\begin{rem} 
It is not clear to the author whether
$H_*(u)$ corresponds to the $U$-map (as defined in \cite{Gutt_Hutchings}) or not. 
\end{rem}

\subsection{Conjectural relation to Gutt-Hutchings capacities}

For any Liouville domain $(X,\lambda)$, 
Gutt-Hutchings \cite{Gutt_Hutchings} defined a sequence 
$(c^{\GH}_k(X,\lambda))_{k \ge 1}$ called Gutt-Hutchings capacities.
See \cite{Gutt_Hutchings} Definitions 4.1 and 4.4 for the definition of the capacities for general Liouville domains.

Let $n$ be a positive integer, and 
$X$ be a star-shaped domain in $\C^n$. 
We abbreviate $c^{\GH}_k(X, \lambda_0)$ by $c^{\GH}_k(X)$. 
For any $0 < a \le \infty$, let 
$\CH^a_*(X):= \varprojlim_{\delta \to 0+} \SH^{S^1,[\delta,a)}_*(X)$. 
There exists a natural map 
$\delta: \CH^\infty_{*-n+1}(X) \to  H_*(X, \del X) \otimes H^{S^1}_*(\pt)$; 
see \cite{Gutt_Hutchings} Section 3. 

For any $a \in \R_{>0}$, let 
$i^a: \CH^a_*(X) \to \CH^\infty_*(X)$ 
be the natural map. Then 
\[ 
c^{\GH}_k(X) = \inf\{ a \mid (i^a)_{2k+1} \ne 0\} 
\] 
for any $k \ge 1$. This follows from the following facts: 
\begin{itemize} 
\item $\delta: \CH^\infty_{n+1}(X) \to H_{2n}(X, \del X) \otimes H^{S^1}_0(\pt)$ is an isomorphism. 
\item $U: \CH^\infty_{n+2k+1}(X) \to \CH^\infty_{n+2k-1}(X)$ is an isomorphism for any $k \ge 1$; 
see \cite{Gutt_Hutchings} Section 6.3 for the definition of the $U$-map. 
\end{itemize} 
It is now clear that Conjecture \ref{conj_main} implies the following conjecture. 

\begin{conj}\label{conj_capacity}
For any $\Omega \in \mca{S}^2$ and $k \ge 1$, 
there holds $c_k(\Omega)=c^{\GH}_k(X_\Omega)$. 
\end{conj} 

\subsection{Computations for concave and (weakly) convex domains} 

In this subsection, we compute the capacities $c_k(\Omega)$ when $\Omega \in \mca{S}^2$ is concave or weakly convex. 

We say $\Omega$ is weakly convex if it is a convex subset of $\R^2$,  
and $\Omega$ is concave if $U_\Omega = (\R_{>0})^2 \setminus \Omega$ is a convex subset of $\R^2$. 
By comparing our computations with 
formulas by Gutt-Hutchings (Theorems 1.6 and 1.14 in \cite{Gutt_Hutchings}), 
we can verify Conjecture \ref{conj_capacity} 
when $\Omega$ is concave or strongly convex. 
Here we say $\Omega$ is strongly convex if 
$\{(x_1, x_2) \in \R^2 \mid (|x_1|, |x_2|) \in \Omega\}$ 
is a convex subset of $\R^2$. 

Recall that $\bar{U}_\Omega$ denotes the closure of $U_\Omega$. 

\begin{prop}\label{prop_concave} 
If $\Omega \in \mca{S}^2$ is concave, then for any $k \ge 1$ 
\[ 
c_k(\Omega) = \max_{1 \le j \le k} ( \min_{p \in \bar{U}_\Omega}   A_{j, k+1-j}(p) ). 
\] 
\end{prop}
\begin{proof} 
We may assume that $\Omega \in \mca{S}^2_\nice$. 
This is because for any concave $\Omega \in \mca{S}^2$ and for any $\ep>0$, 
there exists $\Omega' \in \mca{S}^2_\nice$ which is concave and satisfies 
$\Omega \subset \Omega' \subset (1+\ep)\Omega$. 

Let $a:=\max_{1 \le j \le k} ( \min_{p \in \bar{U}_\Omega}   A_{j, k+1-j}(p) )$. 
To prove $c_k(\Omega)=a$, 
it is sufficient to show 
$a-\ep \le c_k(\Omega) \le a+\ep$
for any $\ep>0$. 

Let us prove $c_k(\Omega) \le a+\ep$. 
Consider the filtration on $C^{\Omega, [a+\ep,\infty)}_*$ as in Lemma \ref{lem_F_m}. 
Since 
$U_\Omega$ and $U_\Omega(a+\ep:m_1,m_2)$ are both convex, we obtain 
\[ 
H_*(U_\Omega, U_\Omega(a+\ep:m_1,m_2)) \cong 
\begin{cases} 
H_*(\pt) &( U_\Omega(a+\ep:m_1,m_2) = \emptyset),  \\ 
0 &(U_\Omega(a+\ep:m_1,m_2) \ne \emptyset). 
\end{cases}
\] 
Moreover, if 
$m_1+m_2 \le k+1$ then 
$\min_{p \in \bar{U}_\Omega} A_{m_1,m_2}(p) \le a$, 
thus $U_\Omega(a+\ep:m_1,m_2) \ne \emptyset$. 
By Remark \ref{rem_E_1}, 
if $E^1_{p,q} \ne 0$ then 
$q \in \{p-1, p\}$ and $p \ge k+2$, 
thus $q+p \ge 2k+3$. 
Hence $H_{\le 2k+2}(C^{\Omega, [a+\ep,\infty)})=0$. 
This implies that 
$H_{2k+1}(C^{\Omega,[\delta, a+\ep)}) \to H_{2k+1}(C^{\Omega, [\delta,\infty)})$ 
is isomorphic for any $\delta \in (0, a+\ep)$, 
thus $c_k(\Omega) \le a+\ep$. 

Let us prove $c_k(\Omega) \ge a-\ep$. 
Take $\delta>0$ sufficiently close to $0$ and 
consider the filtration on $C^{\Omega, [\delta, a-\ep)}_*$ as in Lemma \ref{lem_F_m}. 
Then 
\[ 
E^1_{p,p-1+j} \cong \bigoplus_{\substack{m_1+m_2=p \\ m_1,m_2>0}}  H_0(U_\Omega(a-\ep:m_1,m_2), U_\Omega(\delta: m_1,m_2)) \otimes H_j(S^1). 
\] 
This is because $H_0(U_\Omega(a-\ep:m_1,m_2), U_\Omega(\delta: m_1,m_2))=0$ unless $m_1, m_2>0$. 
Moreover 
$\del_{E^1}: E^1_{p, p-1} \to E^1_{p-1, p-1}$ is given by
\[ 
(\del_{E_1} x)_{m_1, m_2, 1} =  m_1 \cdot x_{m_1, m_2+1, 0} - m_2 \cdot x_{m_1+1, m_2, 0}. 
\] 
There exists $1 \le j \le k$ such that 
$a-\ep < \min_{p \in \bar{U}_\Omega} A_{j, k+1-j}(p)$, 
then 
$U_\Omega(a-\ep: j, k+1-j) = \emptyset$. 
This implies
\[ 
\mathrm{Ker}( \del_{E^1}: E^1_{k+1,k} \to E^1_{k,k})=0, 
\] 
then $H_{2k+1}(C^{\Omega, [\delta, a-\ep)})=0$. 
This implies $c_k(\Omega) \ge a-\ep$. 
\end{proof} 

\begin{prop}\label{prop_convex} 
If $\Omega \in \mca{S}^2$ is weakly convex, then for any $k \ge 1$ 
\[ 
c_k(\Omega) = \min_{0 \le j \le k} ( \max_{p \in \Omega} A_{j, k-j}(p)). 
\] 
\end{prop} 
\begin{proof} 
We may assume that $\Omega \in \mca{S}^2_\nice$. 
This is because for any weakly convex $\Omega \in \mca{S}^2$ and for any $\ep>0$, 
there exists $\Omega' \in \mca{S}^2_\nice$ which is weakly convex and satisfies 
$\Omega \subset \Omega' \subset (1+\ep)\Omega$. 

Let $a:=\min_{0 \le j \le k} ( \max_{p \in \Omega} A_{j, k-j}(p))$. 
It is sufficient to show that, for any $\ep>0$ there holds 
$a-\ep \le c_k(\Omega) \le a+\ep$. 

Let us prove $c_k(\Omega) \le a+\ep$. 
Take $j \in \{0, \ldots, k\}$ so that 
$\max_{p \in \Omega} A_{j,k-j}(p)=a$. 
We fix such $j$ in the following argument. 

There exists $\Omega' \in \mca{S}^2$ which is concave and 
\[ 
\Omega \subset \Omega' \subset \{ p \in (\R_{\ge 0})^2 \mid A_{j, k-j}(p) \le a+\ep\}. 
\] 
Then 
\[ 
c_k(\Omega) \le c_k(\Omega') \le \max_{1 \le i \le k} \big( \min_{A_{j, k-j}(p) \ge a+\ep}  A_{i, k+1-i}(p) \big),
\] 
where the second inequality follows from Proposition \ref{prop_concave}. 
Thus it is sufficient to show 
\begin{equation}\label{eqn_j_k_j}
\min_{A_{j,k-j}(p) \ge a+\ep}   A_{i, k+1-i} (p) \le a+\ep
\end{equation} 
for any $i \in \{1, \ldots, k\}$. 
When $j=0$, the LHS is equal to $\frac{(k+1-i)(a+\ep)}{k}$. 
When $j=k$, the LHS is equal to $\frac{i(a+\ep)}{k}$. 
Thus (\ref{eqn_j_k_j}) holds when $j=0$ or $j=k$. 

Let us consider the case $0<j<k$. 
By $(i-j) + (k+1-i) -(k-j)=1$, we obtain 
\[ 
\min\{ i-j,   (k+1-i)-(k-j)\} \le 0 \iff \min\{ i/j, (k+1-i)/(k-j) \} \le 1. 
\] 
Then 
\[ 
\min_{A_{j, k-j}(p) \ge a+\ep}  A_{i, k+1-i}(p)  = \min \{i/j, (k+1-i)/(k-j) \} \cdot (a+\ep) \le a+\ep. 
\] 
This completes the proof of $c_k(\Omega) \le a+\ep$.  

Let us prove $c_k(\Omega) \ge a-\ep$. 
It is sufficient to show that the image of 
\begin{equation}\label{eqn_convex} 
H_{2k+1}(C^{\Omega, [\delta, a-\ep)}) \to H_{2k+1}(C^{\Omega, [\delta, \infty)})
\end{equation} 
is zero for any $\delta>0$ sufficiently close to $0$. 

Let us first notice that for any $j \in \Z$ 
\[ 
H_*(\bar{U}_\Omega(a-\ep:j, k+1-j), \bar{U}_\Omega(\delta:j, k+1-j))
\cong
H_*(\del_+\Omega(a-\ep:j,k+1-j), \del_+\Omega(\delta:j,k+1-j))
\]
by Lemma \ref{lem_barU}. 
Then we have the following observations: 
\begin{itemize} 
\item[(a):] $H_*(\bar{U}_\Omega(a-\ep:j, k+1-j), \bar{U}_\Omega(\delta:j, k+1-j))=0$ unless $*=0$.
\item[(b):] Any element of $H_0(\bar{U}_\Omega(a-\ep:j, k+1-j), \bar{U}_\Omega(\delta:j, k+1-j))$ 
can be written as $[a^1p_1+a^2p_2]$ with $a^1, a^2 \in \Q$, where 
$p_1:=(\rho(0), 0)$ and $p_2:=(0, \rho(\pi/2))$.  
\end{itemize} 
(a) holds since $\Omega$ is convex and $a-\ep < a \le \max_{p \in \Omega} A_{j, k+1-j}(p)$. 
(b) holds since $\Omega$ is convex. 

Now consider the filtration on $C^{\Omega, [\delta, a-\ep)}_*$ as in Lemma \ref{lem_F_m}. 
By (a) and Remark \ref{rem_E_1}, 
if 
$E^1_{p,q} \ne 0$ and $p+q=2k+1$, 
then $p=k+1$, $q=k$. 
Moreover there exists a natural isomorphism
\[ 
E^1_{k+1,k} \cong \bigoplus_{j \in \Z}  H_0(\bar{U}_\Omega(a-\ep:j, k+1-j), \bar{U}_\Omega(\delta:j, k+1-j)) \otimes  H_0(S^1). 
\]  
Note that we can replace $U_\Omega$ with $\bar{U}_\Omega$ due to Lemma \ref{lem_barU}. 

We are going to prove the following claim: 
\begin{quote} 
If $x= \sum_{j \in \Z}  x_{j, k+1-j} \otimes e_0 \in E^1_{k+1,k}$ satisfies 
$\del_{E^1}(x)=0$, then $x_{j, k+1-j}=0$ for any $1 \le j \le k$. 
\end{quote} 
By (b), for any $j \in \Z$ there exist $a^1_j, a^2_j \in \Q$ such that 
\[ 
x_{j, k+1-j} = [a^1_j p_1 + a^2_j p_2] \in
H_0(\bar{U}_\Omega(a-\ep:j, k+1-j), \bar{U}_\Omega(\delta: j, k+1-j)). 
\] 
Let us prove $[a^1_jp_1]=0$ for $0 \le j \le k$. 
$[a^1_0p_1]=0$ since $[p_1]=0$ in 
$H_0(\bar{U}_\Omega(a-\ep:0, k+1), \bar{U}_\Omega(\delta:0,k+1))$. 
Then it is sufficient to show that 
if $0 \le j \le k-1$ and $[a^1_jp_1]=0$ then $[a^1_{j+1}p_1]=0$. 
This follows from 
$(\del_{E^1}x)_{j, k-j}=0$, $k-j \ne 0$ and the following claim: 
$[\alpha p_1 + \beta p_2]=0 \implies [\alpha p_1]=[\beta p_2]=0$ in 
$H_0(\bar{U}_\Omega(a-\ep:j,k-j), \bar{U}_\Omega(\delta:j,k-j))$. 
This claim holds since for 
$c \in C^0([0,1], \bar{U}_\Omega)$ 
such that $c(0)=p_1$ and $c(1)=p_2$, 
there holds 
\[ 
\max_{t \in [0,1]} A_{j, k-j}(c(t))  \ge \max_{p \in \Omega} A_{j,k-j}(p) \ge a > a-\ep. 
\] 
Now we have proved that $[a^1_jp_1]=0$ for any $0 \le j \le k$. 
By similar arguments, one can prove that $[a^2_j p_2]=0$ for any $1 \le j \le k+1$. 
Then, for any $1 \le j \le k$, we obtain 
$x_{j, k+1-j} = [ a^1_j p_1 + a^2_j p_2] = 0$. 
This finishes the proof of the claim. 

Finally, consider the filtration on $C^{\Omega, [\delta, \infty)}_*$ as in Lemma \ref{lem_F_m}. 
As in the proof of Proposition \ref{prop_plus} (i), there exist natural isomorphisms
\[ 
E^1_{k+1,k} \cong \bigoplus_{j \in \Z }  H_0(\bar{U}_\Omega,  \bar{U}_\Omega(\delta:j, k+1-j))  \otimes  H_0(S^1), 
\]   
and
\[ 
H_0(\bar{U}_\Omega,  \bar{U}_\Omega(\delta:j, k+1-j)) 
\cong 
\begin{cases} 
H_0(\pt) &(1 \le j \le k),  \\ 0 &(\text{otherwise}). 
\end{cases}
\] 
Thus the above claim implies that the image of (\ref{eqn_convex}) is zero. 
This completes the proof of $c_k(\Omega) \ge a-\ep$. 
\end{proof} 

\begin{cor}\label{cor_asymptotics} 
For any $\Omega \in \mca{S}^2$, 
\[ 
\lim_{k \to \infty} \frac{c_k(\Omega)}{k} = \max_{(x_1, x_2) \in \Omega} \min\{ x_1, x_2\}. 
\] 
\end{cor} 
\begin{proof} 
Let $a:=\max_{(x_1, x_2) \in \Omega} \min\{ x_1, x_2\}$. 
It is sufficient to show that, for any $\ep>0$ there holds 
$\limsup_{k \to \infty} \frac{c_k(\Omega)}{k} \le a+\ep$ 
and 
$\liminf_{k \to \infty} \frac{c_k(\Omega)}{k} \ge a-\ep$. 

There exists $\Omega' \in \mca{S}^2$ which is concave and 
\[ 
\Omega \subset \Omega' \subset \{ (x_1, x_2) \in (\R_{\ge 0})^2 \mid \min \{x_1, x_2\} \le  a+ \ep \}. 
\] 
Then, for any $k \ge 1$ 
\[ 
c_k(\Omega) \le c_k(\Omega') \le (k+1)(a+\ep)
\] 
where the second inequality follows from Proposition \ref{prop_concave}. 
Thus $\limsup_{k \to \infty} \frac{c_k(\Omega)}{k} \le a+\ep$. 

There exists 
$(x_1, x_2)$ and $\Omega' \in \mca{S}^2$
such that 
$\min \{x_1, x_2\} \ge a-\ep$, 
$(x_1, x_2) \in \Omega' \subset \Omega$ 
and $\Omega'$ is weakly convex. 
Then
\[ 
c_k(\Omega)  \ge c_k(\Omega') \ge k(a-\ep), 
\] 
where the second inequality follows from Proposition \ref{prop_convex}. 
Then we obtain $c_k(\Omega) \ge k(a-\ep)$ for any $k \ge 1$, 
which implies 
$\liminf_{k \to \infty} \frac{c_k(\Omega)}{k} \ge a-\ep$. 
\end{proof}

\end{document}